\documentclass[oneside,reqno,english]{amsart}
\usepackage{textcomp}
\usepackage[latin9]{inputenc}
\usepackage{mathtools}
\usepackage{amsbsy}
\usepackage{amstext}
\usepackage{amsthm}
\usepackage{amssymb}
\usepackage{geometry}
\usepackage{setspace}

\makeatletter
\theoremstyle{plain}
\newtheorem{lem}{\protect\lemmaname}
\theoremstyle{plain}
\newtheorem{thm}{\protect\theoremname}
\theoremstyle{plain}
\newtheorem{cor}{\protect\corollaryname}

\@ifundefined{date}{}{\date{}}
\PassOptionsToPackage{reqno}{amsmath}

\usepackage{mathtools}

\makeatletter
\AtBeginDocument{\tagsleft@false} % ensures numbers appear on the right
\makeatother

\usepackage{chngcntr}
\counterwithout{equation}{section}

\usepackage{microtype}   % improves text spacing and kerning
\usepackage{setspace}    % control line spacing
\makeatother

\usepackage{babel}
\providecommand{\corollaryname}{Corollary}
\providecommand{\lemmaname}{Lemma}
\providecommand{\theoremname}{Theorem}

\begin{document}
\title{Lower Bounds for Densities of\linebreak{}
Transcendental Gamma-Function Derivatives}
\author{Michael R. Powers}
\address{Department of Finance, School of Economics and Management, Tsinghua
University, Beijing, China 100084}
\email{powers@sem.tsinghua.edu.cn}
\date{April 13, 2026}
\begin{abstract}
In recent work, we showed that for all $q\in\tfrac{1}{2}\mathbb{Z}\setminus\mathbb{Z}_{\leq0}$
the sequence $\left\{ \Gamma^{\left(n\right)}\left(q\right)\right\} _{n\geq1}$
contains transcendental elements infinitely often, with the density
of transcendental $\Gamma^{\left(n\right)}\left(q\right)$ among $n\in\left\{ 1,2,\ldots,N\right\} $
bounded below by $\beta\left(N\right)=\max\left\{ 0,\sqrt{N}-5/2\right\} /N$.
For both fixed and variable $n$, we now study the transcendence of
$\Gamma^{\left(n\right)}\left(q\right)$ at both positive lattice
points $q=m\in\left\{ 1,2,\ldots\right\} $ and rationally shifted
lattice points $q=\widetilde{m}\in\left\{ \kappa,\pm1+\kappa,\pm2+\kappa,\ldots\right\} $
(for $\kappa\in\left(0,1\right)\cap\mathbb{Q}$ such that $\Gamma\left(\kappa\right)$
is transcendental). For $n\in\mathbb{Z}_{\geq2}$, we find there are
at most $n-1$ algebraic $\Gamma^{\left(n\right)}\left(m\right)$,
and for $n\in\mathbb{Z}_{\geq1}$, there are at most $n$ algebraic
$\Gamma^{\left(n\right)}\left(\widetilde{m}\right)$ for each one-sided
shifted lattice (i.e., $\widetilde{m}\geq\kappa$ or $\widetilde{m}\leq\kappa$).
This implies that the density of transcendental $\Gamma^{\left(n\right)}\left(m\right)$
among $m\in\left\{ 1,2,\ldots,M\right\} $ has lower bound $\beta_{n}\left(M\right)=1-\min\left\{ n-1,M\right\} /M$,
whereas the density of transcendental $\Gamma^{\left(n\right)}\left(\widetilde{m}\right)$
among either $\widetilde{m}\in\left\{ \kappa,1+\kappa,2+\kappa,\ldots,M+\kappa\right\} $
or $\widetilde{m}\in\left\{ \kappa,-1+\kappa,-2+\kappa,\ldots,-M+\kappa\right\} $
has lower bound $\widetilde{\beta}_{n}\left(M\right)=1-\min\left\{ n,M+1\right\} /\left(M+1\right)$.
Allowing $n$ to vary, we derive lower bounds for the bivariate densities
of both transcendental $\Gamma^{\left(n\right)}\left(m\right)$ among
$\left(n,m\right)\in\left\{ 2,3,\ldots,N\right\} \times\left\{ 1,2,\ldots,M\right\} $
and transcendental $\Gamma^{\left(n\right)}\left(\widetilde{m}\right)$
among $\left(n,\widetilde{m}\right)\in\left\{ 1,2,\ldots,N\right\} \times\left\{ \kappa,1+\kappa,2+\kappa,\ldots,M+\kappa\right\} $
and $\left(n,\widetilde{m}\right)\in\left\{ 1,2,\ldots,N\right\} \times\left\{ \kappa,-1+\kappa,-2+\kappa,\ldots,-M+\kappa\right\} $.
\end{abstract}

\maketitle

\section{Introduction}

\begin{singlespace}
\noindent Comparatively little is known about the arithmetic properties
of Gamma-function derivatives evaluated at arbitrary points $q\in\mathbb{Q}\setminus\mathbb{Z}_{\leq0}$.\footnote{See Rivoal {[}1{]} and Fischler and Rivoal {[}2{]}.}
In recent work, we showed that for all $q\in\tfrac{1}{2}\mathbb{Z}\setminus\mathbb{Z}_{\leq0}$
the sequence $\left\{ \Gamma^{\left(n\right)}\left(q\right)\right\} _{n\geq1}$
contains transcendental elements infinitely often (see Theorem 1 of
Powers {[}3{]}) and the density of transcendental $\Gamma^{\left(n\right)}\left(q\right)$
among $n\in\left\{ 1,2,\ldots,N\right\} $ is bounded below by
\[
\beta\left(N\right)=\dfrac{\max\left\{ 0,\sqrt{N}-5/2\right\} }{N}
\]
\begin{equation}
\Longrightarrow\beta\left(N\right)\asymp\dfrac{1}{\sqrt{N}}\textrm{ as }N\rightarrow\infty
\end{equation}
(see Theorem 2 of Powers {[}3{]}). For both fixed and variable $n$,
we now study the transcendence of $\Gamma^{\left(n\right)}\left(q\right)$
at both positive lattice points $q=m\in\left\{ 1,2,\ldots\right\} $
and rationally shifted lattice points $q=\widetilde{m}\in\left\{ \kappa,\pm1+\kappa,\pm2+\kappa,\ldots\right\} $
for $\kappa\in\mathcal{K}$, where $\mathcal{K}=\left\{ \kappa:\left(\kappa\in\left(0,1\right)\cap\mathbb{Q}\right)\wedge\left(\Gamma\left(\kappa\right)\textrm{ is transcendental}\right)\right\} $.
(Currently, the only shift values known to satisfy these conditions
are $\kappa\in\mathcal{K}^{*}=\left\{ 1/6,1/4,1/3,1/2,2/3,3/4,5/6\right\} $;
however, it is generally believed that $\Gamma\left(q\right)$ is
transcendental for all $q\in\left(0,1\right)\cap\mathbb{Q}$.)
\end{singlespace}

\begin{singlespace}
In Section 2, we fix $n$ and derive strong upper bounds for the total
number of algebraic $\Gamma^{\left(n\right)}\left(m\right)$ and $\Gamma^{\left(n\right)}\left(\widetilde{m}\right)$,
respectively. Specifically, we find that for $n\in\mathbb{Z}_{\geq2}$
there are at most $n-1$ algebraic $\Gamma^{\left(n\right)}\left(m\right)$,
and for $n\in\mathbb{Z}_{\geq1}$ there are at most $n$ algebraic
$\Gamma^{\left(n\right)}\left(\widetilde{m}\right)$ for each one-sided
shifted lattice (i.e., $\widetilde{m}\geq\kappa$ or $\widetilde{m}\leq\kappa$).
These results immediately imply that the density of transcendental
$\Gamma^{\left(n\right)}\left(m\right)$ among $m\in\left\{ 1,2,\ldots,M\right\} =\mathcal{L}\left(M\right)$
has lower bound
\[
\beta_{n}\left(M\right)=1-\dfrac{\min\left\{ n-1,M\right\} }{M}
\]
\[
\Longrightarrow\beta_{n}\left(M\right)\asymp1-\dfrac{n-1}{M}\rightarrow1\textrm{ as }M\rightarrow\infty,
\]
whereas the density of transcendental $\Gamma^{\left(n\right)}\left(\widetilde{m}\right)$
among either $\widetilde{m}\in\left\{ \kappa,1+\kappa,2+\kappa,\ldots,M+\kappa\right\} =\mathcal{L}_{\kappa}^{\left(+\right)}\left(M\right)$
or $\widetilde{m}\in\left\{ \kappa,-1+\kappa,-2+\kappa,\ldots,-M+\kappa\right\} =\mathcal{L}_{\kappa}^{\left(-\right)}\left(M\right)$
has lower bound
\[
\widetilde{\beta}_{n}\left(M\right)=1-\dfrac{\min\left\{ n,M+1\right\} }{M+1}
\]
\[
\Longrightarrow\widetilde{\beta}_{n}\left(M\right)\asymp1-\dfrac{n}{M}\rightarrow1\textrm{ as }M\rightarrow\infty.
\]

In Section 3, we allow $n$ to vary, evaluating the bivariate densities
of both transcendental $\Gamma^{\left(n\right)}\left(m\right)$ among
$\left(n,m\right)\in\left\{ 2,3,\ldots,N\right\} \times\mathcal{L}\left(M\right)$
and transcendental $\Gamma^{\left(n\right)}\left(\widetilde{m}\right)$
among $\left(n,\widetilde{m}\right)\in\left\{ 1,2,\ldots,N\right\} \times\mathcal{L}_{\kappa}^{\left(+\right)}\left(M\right)$
and $\left(n,\widetilde{m}\right)\in\left\{ 1,2,\ldots,N\right\} \times\mathcal{L}_{\kappa}^{\left(-\right)}\left(M\right)$.
In the former case, this analysis yields lower bounds of 
\[
\beta\left(N,M\right)=\begin{cases}
\left(M-1\right)/\left[2\left(N-1\right)\right], & M\leq N-1\\
1-N/\left(2M\right), & M>N-1
\end{cases}
\]
\[
\Longrightarrow\beta\left(N,M\right)\asymp\begin{cases}
M/\left(2N\right), & M/N\leq1\\
1-N/\left(2M\right), & M/N>1
\end{cases}\textrm{ as }N,M\rightarrow\infty,
\]
whereas in the latter case the lower bounds are given by

\[
\widetilde{\beta}\left(N,M\right)=\begin{cases}
M/\left(2N\right), & M+1\leq N\\
1-\left(N+1\right)/\left[2\left(M+1\right)\right], & M+1>N
\end{cases}
\]
\[
\Longrightarrow\widetilde{\beta}\left(N,M\right)\asymp\begin{cases}
M/\left(2N\right), & M/N\leq1\\
1-N/\left(2M\right), & M/N>1
\end{cases}\textrm{ as }N,M\rightarrow\infty.
\]

\end{singlespace}

\section{Lower Bounds with Fixed $n$, Variable $m$}

\subsection{Positive Lattice Points}

\begin{singlespace}
\phantom{}

\medskip{}

\end{singlespace}

\begin{singlespace}
\noindent The present subsection addresses the transcendence of $\Gamma^{\left(n\right)}\left(m\right)$
at positive lattice points $m\in\left\{ 1,2,\ldots\right\} =\mathcal{L}\left(\infty\right)$
for fixed $n\in\mathbb{Z}_{\geq2}$. Specifically, Theorem 1 provides
an upper bound for the total number of algebraic $\Gamma^{\left(n\right)}\left(m\right)$
in §2.1.3, whereas Corollary 1 gives the corresponding lower bound
for the density of transcendental $\Gamma^{\left(n\right)}\left(m\right)$
among $m\in\mathcal{L}\left(M\right)$ in §2.1.4. Mathematical preliminaries
are presented in §2.1.1 and §2.1.2 below.\medskip{}

\noindent\textbf{\emph{2.1.1\enskip{}A Useful Identity}}\medskip{}

\noindent For $m\in\mathbb{Z}_{\geq1}$, the functional equation $\Gamma\left(z+1\right)=z\Gamma\left(z\right)$
gives
\[
\Gamma\left(m+t\right)=\Gamma\left(1+t\right)\prod_{u=1}^{m-1}\left(u+t\right)
\]
\[
=\left(m-1\right)!\Gamma\left(1+t\right)\prod_{u=1}^{m-1}\left(1+\frac{t}{u}\right)
\]
\begin{equation}
=\left(m-1\right)!\Gamma\left(1+t\right)\sum_{v=0}^{m-1}e_{v}\left(1,\frac{1}{2},\dots,\frac{1}{m-1}\right)t^{v},
\end{equation}
where ${\textstyle \prod_{u=1}^{m-1}\left(u+t\right)\coloneqq1}$
for $m=1$ and $e_{v}\left(x_{1},x_{2},\dots,x_{k}\right)$ denotes
the $v^{\textrm{th}}$ elementary symmetric polynomial with $e_{0}\left(x_{1},x_{2},\dots,x_{k}\right)\coloneqq1$,
$e_{0}\left(\emptyset\right)\coloneqq1$, $e_{v}\left(x_{1},x_{2},\dots,x_{k}\right)\coloneqq0$
for $v>k$, and $e_{v}\left(\emptyset\right)\coloneqq0$ for $v>0$
(in particular, the argument of $e_{v}\left(1,1/2,\dots,1/\left(m-1\right)\right)$
is treated as $\emptyset$ for $m=1$). Differentiating (2) $n$ times
with respect to $t$ and then evaluating at $t=0$ gives
\[
\Gamma^{\left(n\right)}\left(m\right)=\left(m-1\right)!\sum_{\ell=0}^{n}\dfrac{n!}{\ell!}e_{n-\ell}\left(1,\frac{1}{2},\dots,\frac{1}{m-1}\right)\Gamma^{\left(\ell\right)}\left(1\right)
\]
\begin{equation}
=\sum_{\ell=0}^{n}T_{n,\ell}\left(m\right)\Gamma^{\left(\ell\right)}\left(1\right),
\end{equation}
where
\[
T_{n,\ell}\left(m\right)=\left(m-1\right)!\frac{n!}{\ell!}e_{n-\ell}\left(1,\frac{1}{2},\dots,\frac{1}{m-1}\right)\in\mathbb{Q}
\]
and $\Gamma^{\left(0\right)}\left(1\right)=\Gamma\left(1\right)=1$.
\end{singlespace}

\begin{singlespace}
The identity in (3), which holds for all $n\in\mathbb{Z}_{\geq0}$
and $m\in\mathbb{Z}_{\geq1}$, is fundamental to the present analysis.
In particular, we consider a system of $n\in\mathbb{Z}_{\geq2}$ equations
of this form, based on a set of arbitrary $m$ values, $\mathcal{M}=\left\{ m_{1},m_{2},\ldots,m_{n}\right\} $
with $m_{r}\in\mathbb{Z}_{\geq1}$ for all $r\in\left\{ 1,2,\ldots,n\right\} $
and $m_{1}<m_{2}<\ldots<m_{n}$. Specifically,
\begin{equation}
\left[\begin{array}{c}
\Gamma^{\left(n\right)}\left(m_{1}\right)\\
\Gamma^{\left(n\right)}\left(m_{2}\right)\\
\vdots\\
\Gamma^{\left(n\right)}\left(m_{n}\right)
\end{array}\right]=\mathbf{T}_{n}\left(\mathcal{M}\right)\left[\begin{array}{c}
\Gamma^{\left(1\right)}\left(1\right)\\
\Gamma^{\left(2\right)}\left(1\right)\\
\vdots\\
\Gamma^{\left(n\right)}\left(1\right)
\end{array}\right]+\left[\begin{array}{c}
T_{n,0}\left(m_{1}\right)\\
T_{n,0}\left(m_{2}\right)\\
\vdots\\
T_{n,0}\left(m_{n}\right)
\end{array}\right],
\end{equation}
with $n\times n$ coefficient matrix
\[
\mathbf{T}_{n}\left(\mathcal{M}\right)=\left[\begin{array}{cccc}
T_{n,1}\left(m_{1}\right) & T_{n,2}\left(m_{1}\right) & \cdots & T_{n,n}\left(m_{1}\right)\\
T_{n,1}\left(m_{2}\right) & T_{n,2}\left(m_{2}\right) & \cdots & T_{n,n}\left(m_{2}\right)\\
\vdots & \vdots & \ddots & \vdots\\
T_{n,1}\left(m_{n}\right) & T_{n,2}\left(m_{n}\right) & \cdots & T_{n,n}\left(m_{n}\right)
\end{array}\right].
\]
\medskip{}

\end{singlespace}

\begin{singlespace}
\noindent\textbf{\emph{2.1.2\enskip{}Invertibility of the Coefficient
Matrix}}\medskip{}

\end{singlespace}

\begin{singlespace}
\noindent To study the arithmetic nature of the values $\left\{ \Gamma^{\left(n\right)}\left(m_{1}\right),\Gamma^{\left(n\right)}\left(m_{2}\right),\ldots,\Gamma^{\left(n\right)}\left(m_{n}\right)\right\} $,
we first confirm that the coefficient matrix $\mathbf{T}_{n}\left(\mathcal{M}\right)$
is invertible (i.e., nonsingular) in $\mathbb{Q}$ via the following
two lemmas.
\end{singlespace}
\begin{lem}
For any $n\in\mathbb{Z}_{\geq1}$ and set $\mathcal{M}^{\prime}=\left\{ m_{1}^{\prime},m_{2}^{\prime},\ldots,m_{n}^{\prime}\right\} $
(with $m_{r}^{\prime}\in\mathbb{Z}_{\geq0}$ for all $r\in\left\{ 1,2,\ldots,n\right\} $
and $m_{1}^{\prime}<m_{2}^{\prime}<\ldots<m_{n}^{\prime}$), let $\boldsymbol{x}^{\left(m_{r}^{\prime}\right)}=\left[x_{1},x_{2},\ldots,x_{m_{r}^{\prime}}\right]$
(with fixed $x_{i}\in\mathbb{R}_{>0}$ for $i\in\left\{ 1,2,\ldots,m_{n}^{\prime}\right\} $)
and define the $n\times n$ matrix
\[
\mathbf{E}_{n}\left(\mathcal{M}^{\prime}\right)=\left[\begin{array}{cccc}
1 & e_{1}\left(\boldsymbol{x}^{\left(m_{1}^{\prime}\right)}\right) & \cdots & e_{n-1}\left(\boldsymbol{x}^{\left(m_{1}^{\prime}\right)}\right)\\
1 & e_{1}\left(\boldsymbol{x}^{\left(m_{2}^{\prime}\right)}\right) & \cdots & e_{n-1}\left(\boldsymbol{x}^{\left(m_{2}^{\prime}\right)}\right)\\
\vdots & \vdots & \ddots & \vdots\\
1 & e_{1}\left(\boldsymbol{x}^{\left(m_{n}^{\prime}\right)}\right) & \cdots & e_{n-1}\left(\boldsymbol{x}^{\left(m_{n}^{\prime}\right)}\right)
\end{array}\right].
\]
Then $\det\left(\mathbf{E}_{n}\left(\mathcal{M}^{\prime}\right)\right)>0$.
\end{lem}
\begin{proof}
\begin{singlespace}
\phantom{}
\end{singlespace}

\begin{singlespace}
\medskip{}
\noindent We proceed by induction on $n$, beginning with $n=1$, for which
\[
\mathbf{E}_{1}\left(\mathcal{M}^{\prime}\right)=\left[1\right]\Longrightarrow\det\left(\mathbf{E}_{1}\left(\mathcal{M}^{\prime}\right)\right)=1.
\]
\end{singlespace}

\begin{singlespace}
Now assume the lemma holds for arbitrary $n-1$, and construct a new
$n\times n$ matrix, $\mathbf{F}_{n}\left(\mathcal{M}^{\prime}\right)$,
whose first row is identical to the first row of $\mathbf{E}_{n}\left(\mathcal{M}^{\prime}\right)$,
and whose $r^{\textrm{th}}$ row is given by the arithmetic difference
between the $r^{\textrm{th}}$ and $\left(r-1\right)^{\textrm{th}}$
rows of $\mathbf{E}_{n}\left(\mathcal{M}^{\prime}\right)$; that is,
\[
\left[\begin{array}{cccc}
0 & e_{1}\left(\boldsymbol{x}^{\left(m_{r}^{\prime}\right)}\right)-e_{1}\left(\boldsymbol{x}^{\left(m_{r-1}^{\prime}\right)}\right) & \cdots & e_{n-1}\left(\boldsymbol{x}^{\left(m_{r}^{\prime}\right)}\right)-e_{n-1}\left(\boldsymbol{x}^{\left(m_{r-1}^{\prime}\right)}\right)\end{array}\right]
\]
for $r\in\left\{ 2,3,\ldots,n\right\} $. This new matrix has first
column $\left[\begin{array}{cccc}
1 & 0 & \cdots & 0\end{array}\right]^{\mathrm{T}}$ and $\det\left(\mathbf{F}_{n}\left(\mathcal{M}^{\prime}\right)\right)=\det\left(\mathbf{E}_{n}\left(\mathcal{M}^{\prime}\right)\right)$.
\end{singlespace}

\begin{singlespace}
Using the first column of $\mathbf{F}_{n}\left(\mathcal{M}^{\prime}\right)$
to compute $\det\left(\mathbf{F}_{n}\left(\mathcal{M}^{\prime}\right)\right)$
by the method of minors, we see that
\[
\det\left(\mathbf{F}_{n}\left(\mathcal{M}^{\prime}\right)\right)=\det\left(\mathbf{D}_{n-1}\left(\mathcal{M}^{\prime}\right)\right),
\]
where $\mathbf{D}_{n-1}\left(\mathcal{M}^{\prime}\right)$ is the
$(n-1)\times(n-1)$ matrix{\small
\[
\mathbf{D}_{n-1}\left(\mathcal{M}^{\prime}\right)=\left[\begin{array}{cccc}
e_{1}\left(\boldsymbol{x}^{\left(m_{2}^{\prime}\right)}\right)-e_{1}\left(\boldsymbol{x}^{\left(m_{1}^{\prime}\right)}\right) & e_{2}\left(\boldsymbol{x}^{\left(m_{2}^{\prime}\right)}\right)-e_{2}\left(\boldsymbol{x}^{\left(m_{1}^{\prime}\right)}\right) & \cdots & e_{n-1}\left(\boldsymbol{x}^{\left(m_{2}^{\prime}\right)}\right)-e_{n-1}\left(\boldsymbol{x}^{\left(m_{1}^{\prime}\right)}\right)\\
e_{1}\left(\boldsymbol{x}^{\left(m_{3}^{\prime}\right)}\right)-e_{1}\left(\boldsymbol{x}^{\left(m_{2}^{\prime}\right)}\right) & e_{2}\left(\boldsymbol{x}^{\left(m_{3}^{\prime}\right)}\right)-e_{2}\left(\boldsymbol{x}^{\left(m_{2}^{\prime}\right)}\right) & \cdots & e_{n-1}\left(\boldsymbol{x}^{\left(m_{3}^{\prime}\right)}\right)-e_{n-1}\left(\boldsymbol{x}^{\left(m_{2}^{\prime}\right)}\right)\\
\vdots & \vdots & \ddots & \vdots\\
e_{1}\left(\boldsymbol{x}^{\left(m_{n}^{\prime}\right)}\right)-e_{1}\left(\boldsymbol{x}^{\left(m_{n-1}^{\prime}\right)}\right) & e_{2}\left(\boldsymbol{x}^{\left(m_{n}^{\prime}\right)}\right)-e_{2}\left(\boldsymbol{x}^{\left(m_{n-1}^{\prime}\right)}\right) & \cdots & e_{n-1}\left(\boldsymbol{x}^{\left(m_{n}^{\prime}\right)}\right)-e_{n-1}\left(\boldsymbol{x}^{\left(m_{n-1}^{\prime}\right)}\right)
\end{array}\right].
\]
}{\small\par}
\end{singlespace}

\begin{singlespace}
\noindent From the recurrence
\begin{equation}
e_{c}\left(\boldsymbol{x}^{\left(m_{r}^{\prime}\right)}\right)=\begin{cases}
e_{c}\left(\boldsymbol{x}^{\left(m_{r}^{\prime}-1\right)}\right)+x_{m_{r}^{\prime}}e_{c-1}\left(\boldsymbol{x}^{\left(m_{r}^{\prime}-1\right)}\right), & m_{r}^{\prime}\geq1\\
e_{c}\left(\emptyset\right), & m_{r}^{\prime}=0
\end{cases},
\end{equation}
where $r$ and $c$ are the row and column indices, respectively,
we find further that the generic ($r,c$) entry of $\mathbf{D}_{n-1}\left(\mathcal{M}^{\prime}\right)$
can be written as
\[
D_{r,c}=\sum_{j=m_{r}^{\prime}+1}^{m_{r+1}^{\prime}}x_{j}e_{c-1}\left(\boldsymbol{x}^{\left(j-1\right)}\right).
\]
Now define the $\left(n-1\right)\times m_{n}^{\prime}$ matrix $\mathbf{A}$
by its generic ($r,j$) entry
\[
A_{r,j}=\begin{cases}
x_{j}, & m_{r}^{\prime}<j\leq m_{r+1}^{\prime}\\
0, & \text{otherwise}
\end{cases},
\]
and likewise define the $m_{n}^{\prime}\times\left(n-1\right)$ matrix
$\mathbf{B}$ by its generic ($j,c$) entry
\[
B_{j,c}=e_{c-1}\left(\boldsymbol{x}^{\left(j-1\right)}\right).
\]
Since $\mathbf{D}_{n-1}\left(\mathcal{M}^{\prime}\right)=\mathbf{A}\mathbf{B}$,
it follows from the Cauchy-Binet formula that
\[
\det\left(\mathbf{D}_{n-1}\left(\mathcal{M}^{\prime}\right)\right)=\sum_{\substack{\mathcal{S}\subseteq\left\{ 1,2,\ldots,m_{n}^{\prime}\right\} ,\\
|\mathcal{S}|=n-1
}
}\det\left(\mathbf{A}_{*,\mathcal{S}}\right)\det\left(\mathbf{B}_{\mathcal{S},*}\right).
\]
\end{singlespace}

\begin{singlespace}
Given that the supports of the rows of $\mathbf{A}$ are the disjoint
intervals $\left(m_{r}^{\prime},m_{r+1}^{\prime}\right]$, we know
that $\det\left(\mathbf{A}_{*,\mathcal{S}}\right)=0$ unless $\mathcal{S}=\left\{ s_{1}<s_{2}<\cdots<s_{n-1}\right\} $
with $m_{r}^{\prime}<s_{r}\leq m_{r+1}^{\prime}$ for all $r$, in
which case $\mathbf{A}_{*,\mathcal{S}}$ is diagonal with diagonal
entries $x_{s_{1}},x_{s_{2}},\dots,x_{s_{n-1}}$, implying $\det\left(\mathbf{A}_{*,\mathcal{S}}\right)=\prod_{r=1}^{n-1}x_{s_{r}}>0$.
Moreover, for any such $\mathcal{S}$, the matrix $\mathbf{B}_{\mathcal{S},*}$
has entries 
\[
B_{r^{\prime},c^{\prime}}^{\left(\mathcal{S},*\right)}=e_{c^{\prime}-1}\left(\boldsymbol{x}^{\left(s_{r^{\prime}}-1\right)}\right),
\]
with $0\le s_{1}-1<s_{2}-1<\cdots<s_{n-1}-1$, giving $\mathbf{B}_{\mathcal{S},*}$
the same form as $\mathbf{E}_{n-1}\left(\mathcal{M}^{\prime}\right)$,
so that $\det\left(\mathbf{B}_{\mathcal{S},*}\right)>0$ by the induction
hypothesis. Thus, every non-zero summand in the Cauchy-Binet expansion
is strictly positive, and since there exists at least one $\mathcal{S}$
such that $m_{r}^{\prime}<s_{r}\leq m_{r+1}^{\prime}$ for all $r$
(e.g., we can set $s_{r}=m_{r}^{\prime}+1$), it follows that
\[
\det\left(\mathbf{D}_{n-1}\left(\mathcal{M}^{\prime}\right)\right)>0
\]
\[
\Longrightarrow\det\left(\mathbf{E}_{n}\left(\mathcal{M}^{\prime}\right)\right)=\det\left(\mathbf{F}_{n}\left(\mathcal{M}^{\prime}\right)\right)>0.
\]
\end{singlespace}
\end{proof}
\begin{lem}
For any $n\in\mathbb{Z}_{\geq1}$ and set $\mathcal{M}=\left\{ m_{1},m_{2},\ldots,m_{n}\right\} $
(with $m_{r}\in\mathbb{Z}_{\geq1}$ for all $r\in\left\{ 1,2,\ldots,n\right\} $
and $m_{1}<m_{2}<\ldots<m_{n}$), $\mathbf{T}_{n}\left(\mathcal{M}\right)$
is nonsingular over $\mathbb{Q}$.
\end{lem}
\begin{proof}
\begin{singlespace}
\phantom{}
\end{singlespace}

\begin{singlespace}
\medskip{}
\noindent In defining the matrix $\mathbf{T}_{n}\left(\mathcal{M}\right)$,
we used
\[
T_{n,c}\left(m_{r}\right)=\frac{\left(m_{r}-1\right)!n!}{c!}e_{n-c}\left(1,\frac{1}{2},\dots,\frac{1}{m_{r}-1}\right)
\]
as the generic entry for $r,c\in\left\{ 1,2,\ldots,n\right\} $. Now
note that if we multiply each column of this matrix by the non-zero
rational factor $c!/n!$ and then reverse the column order (so that
$\left(m_{r}-1\right)!e_{n-c}\left(1,1/2,\dots,1/\left(m_{r}-1\right)\right)$
becomes $\left(m_{r}-1\right)!e_{c-1}\left(1,1/2,\dots,1/\left(m_{r}-1\right)\right)$),
we transform $\mathbf{T}_{n}\left(\mathcal{M}\right)$ into a row-scaled
version of the matrix $\mathbf{E}_{n}\left(\mathcal{M}^{\prime}\right)$
(of Lemma 1) with $m_{r}^{\prime}=m_{r}-1$ and $x_{s}=1/s$. Since
the two steps of this transformation do not affect nonsingularity,
it follows from Lemma 1 that
\[
\det\left(\mathbf{T}_{n}\left(\mathcal{M}\right)\right)\neq0,
\]
so that $\mathbf{T}_{n}\left(\mathcal{M}\right)$ is nonsingular.
Given that all entries of the matrix lie in $\mathbb{Q}$, we know
that all entries of its inverse, $\mathbf{T}_{n}^{-1}\left(\mathcal{M}\right)$,
lie in $\mathbb{Q}$ as well.
\end{singlespace}
\end{proof}
\medskip{}

\begin{singlespace}
\noindent\textbf{\emph{2.1.3\enskip{}Upper Bound for Number of Algebraic
Values}}\medskip{}

\end{singlespace}

\begin{singlespace}
\noindent Having determined that the coefficient matrix $\mathbf{T}_{n}\left(\mathcal{M}\right)$
is invertible, we now apply this result to the linear system in (4)
to provide an upper bound for the number of algebraic values $\Gamma^{\left(n\right)}\left(m\right)$
for fixed $n\in\mathbb{Z}_{\geq2}$.
\end{singlespace}
\begin{thm}
For $n\in\mathbb{Z}_{\geq2}$,
\[
\#\left(m\in\mathbb{Z}_{\geq1}:\Gamma^{\left(n\right)}\left(m\right)\textrm{ is algebraic}\right)\le n-1.
\]
\end{thm}
\begin{proof}
\begin{singlespace}
\phantom{}
\end{singlespace}

\begin{singlespace}
\medskip{}
\noindent Fixing $n$, consider a set $\mathcal{M}=\left\{ m_{1},m_{2},\ldots,m_{n}\right\} $
as defined above and assume (for purposes of contradiction) that $\Gamma^{\left(n\right)}\left(m_{r}\right)\in\overline{\mathbb{Q}}$
for all $r\in\left\{ 1,2,\ldots,n\right\} $. Given that $\mathbf{T}_{n}\left(\mathcal{M}\right)$
is invertible (by Lemma 2), we then can solve for the column vector
$\left[\Gamma^{\left(1\right)}\left(1\right),\Gamma^{\left(2\right)}\left(1\right),\ldots,\Gamma^{\left(n\right)}\left(1\right)\right]^{\mathrm{T}}$
in (4) as
\begin{equation}
\left[\begin{array}{c}
\Gamma^{\left(1\right)}\left(1\right)\\
\Gamma^{\left(2\right)}\left(1\right)\\
\vdots\\
\Gamma^{\left(n\right)}\left(1\right)
\end{array}\right]=\left(\mathbf{T}_{n}\left(\mathcal{M}\right)\right)^{-1}\left[\begin{array}{c}
\Gamma^{\left(n\right)}\left(m_{1}\right)-T_{n,0}\left(m_{1}\right)\\
\Gamma^{\left(n\right)}\left(m_{2}\right)-T_{n,0}\left(m_{2}\right)\\
\vdots\\
\Gamma^{\left(n\right)}\left(m_{n}\right)-T_{n,0}\left(m_{n}\right)
\end{array}\right],
\end{equation}
where all elements of the column vector on the right-hand side of
(6) are algebraic numbers under the contradiction hypothesis.
\end{singlespace}

\begin{singlespace}
Since the inverse matrix $\left(\mathbf{T}_{n}\left(\mathcal{M}\right)\right)^{-1}$
contains only rational entries, it follows that all elements of\linebreak{}
$\left[\Gamma^{\left(1\right)}\left(1\right),\Gamma^{\left(2\right)}\left(1\right),\ldots,\Gamma^{\left(n\right)}\left(1\right)\right]^{\mathrm{T}}$
must be algebraic. However, at least one element of the pair\linebreak{}
$\left\{ \Gamma^{\left(1\right)}\left(1\right)=-\gamma,\Gamma^{\left(2\right)}\left(1\right)=\gamma^{2}+\zeta\left(2\right)\right\} $
is transcendental (because otherwise, $\Gamma^{\left(2\right)}\left(1\right)-\left(\Gamma^{\left(1\right)}\left(1\right)\right)^{2}=\zeta\left(2\right)=\pi^{2}/6$
would be algebraic), forcing a contradiction that precludes $\Gamma^{\left(n\right)}\left(m_{r}\right)\in\overline{\mathbb{Q}}$
for all $r\in\left\{ 1,2,\ldots,n\right\} $.\footnote{In the special case of $n=2$, the contradiction hypothesis (in particular,
$\Gamma^{\left(2\right)}\left(1\right)\in\overline{\mathbb{Q}}$)
implies $\Gamma^{\left(1\right)}\left(1\right)\notin\overline{\mathbb{Q}}$. } Therefore, we conclude that at most $n-1$ of the elements in $\left\{ \Gamma^{\left(n\right)}\left(m_{1}\right),\Gamma^{\left(n\right)}\left(m_{2}\right),\ldots,\Gamma^{\left(n\right)}\left(m_{n}\right)\right\} $
can be algebraic. Furthermore, given that $\mathcal{M}$ was constructed
of arbitrary $m\in\mathbb{Z}_{\geq1}$, it is impossible to assemble
\emph{any} set of algebraic $\Gamma^{\left(n\right)}\left(m\right)$
with more than $n-1$ distinct $m$.
\end{singlespace}
\end{proof}
\begin{singlespace}
The above result is rather strong in that it demonstrates the transcendence
of $\Gamma^{\left(n\right)}\left(m\right)$ for ``all but a finite
number of $m$'' for each $n\in\mathbb{Z}_{\geq2}$, whereas Theorem
1 of Powers {[}3{]} asserts only the transcendence of $\Gamma^{\left(n\right)}\left(m\right)$
for ``infinitely many $n$'' for each $m\in\mathbb{Z}_{\geq1}$.
In particular, for the lowest value of $n$ addressed by Theorem 1
above (i.e., $n=2$), we now know that $\Gamma^{\left(2\right)}\left(m\right)$
is transcendental for ``all but at most one $m$''.\footnote{It is intriguing to note that if (somehow) Theorem 1 could be extended
to $n=1$, then that would confirm the transcendence of $\gamma$
(the Euler-Mascheroni constant). }\medskip{}

\end{singlespace}

\begin{singlespace}
\noindent\textbf{\emph{2.1.4\enskip{}Lower Bound for Transcendental
Densities}}\medskip{}

\end{singlespace}

\begin{singlespace}
\noindent For fixed $n\in\mathbb{Z}_{\geq2}$, let
\[
\delta_{n}\left(M\right)=\dfrac{\#\left\{ m\in\mathcal{L}\left(M\right):\Gamma^{\left(n\right)}\left(m\right)\textrm{ is transcendental}\right\} }{M},\quad M\in\mathbb{Z}_{\geq1}
\]
denote the density of the set of transcendental $\Gamma^{\left(n\right)}\left(m\right)$
among $m\in\mathcal{L}\left(M\right)$. The following corollary to
Theorem 1 provides a lower bound for this density.
\end{singlespace}
\begin{cor}
For $n\in\mathbb{Z}_{\geq2}$,
\begin{equation}
\delta_{n}\left(M\right)\geq\beta_{n}\left(M\right)=1-\dfrac{\min\left\{ n-1,M\right\} }{M},
\end{equation}
where
\begin{equation}
\beta_{n}\left(M\right)\asymp1-\dfrac{n-1}{M}\;as\;M\rightarrow\infty.
\end{equation}
 
\end{cor}
\begin{proof}
\begin{singlespace}
\phantom{}
\end{singlespace}

\begin{singlespace}
\medskip{}
\noindent For values of $M>n-1$, Theorem 1 implies that the density of algebraic
$\Gamma^{\left(n\right)}\left(m\right)$ among $m\in\mathcal{L}\left(M\right)$
is (weakly) bounded above by $\left(n-1\right)/M$. On the other hand,
for $M\leq n-1$, Theorem 1 does not preclude all $\Gamma^{\left(n\right)}\left(m\right)$
from being algebraic, resulting in the (weak) upper bound of $M/M=1$.
To obtain lower bounds for the transcendental density, we simply take
the complements of these upper bounds, giving
\[
\delta_{n}\left(M\right)\geq\beta_{n}\left(M\right)=\begin{cases}
0/M, & M\leq n-1\\
1-\left(n-1\right)/M, & M>n-1
\end{cases},
\]
from which (7) and (8) immediately follow.
\end{singlespace}
\end{proof}
\begin{singlespace}
Comparing the asymptotic lower bounds of Corollary 1 above and Theorem
2 of Powers {[}3{]} -- given by $\beta_{n}\left(M\right)\asymp1-\left(n-1\right)/M$
and $\beta\left(N\right)\asymp1/\sqrt{N}$ (as indicated in (1)),
respectively -- it is easy to see that the former benefits greatly
from the strength of Theorem 1 above, converging to $1$ in the limit
as $M\rightarrow\infty$ (whereas the latter converges to $0$ in
the limit as $N\rightarrow\infty$). In other words, by analyzing
the problem from the perspective of the $m$ dimension rather than
the $n$ dimension, Corollary 1 is able to demonstrate a dramatically
greater density of transcendental values $\Gamma^{\left(n\right)}\left(m\right)$.
If we assume (as seems reasonable) that density characteristics are
homogeneous and independent over the $m$ dimension, then we may infer
that the weaker bound in Theorem 2 of Powers {[}3{]} can be greatly
improved.
\end{singlespace}

\subsection{Shifted Lattice Points}

\begin{singlespace}
\phantom{}

\medskip{}

\end{singlespace}

\begin{singlespace}
\noindent We now consider the transcendence of $\Gamma^{\left(n\right)}\left(\widetilde{m}\right)$
at rationally shifted lattice points $\widetilde{m}\in\mathcal{L}_{\kappa}^{\left(+\right)}\left(M\right)$
or $\widetilde{m}\in\mathcal{L}_{\kappa}^{\left(-\right)}\left(M\right)$
for fixed $n\in\mathbb{Z}_{\geq1}$ and $\kappa\in\mathcal{K}$. The
present exposition is arranged in a manner similar to Subsection 2.1,
with: (i) Theorem 2 providing an upper bound for the total number
of algebraic $\Gamma^{\left(n\right)}\left(\widetilde{m}\right)$
in §2.2.3; (ii) Corollary 2 giving the corresponding lower bound for
the density of transcendental $\Gamma^{\left(n\right)}\left(\widetilde{m}\right)$
among $\widetilde{m}\in\mathcal{L}_{\kappa}^{\left(+\right)}\left(M\right)$
or $\widetilde{m}\in\mathcal{L}_{\kappa}^{\left(-\right)}\left(M\right)$
in §2.2.4; and (iii) preliminary mathematical analysis presented in
§2.2.1 and §2.2.2.\footnote{Note that the split between §2.2.1 and §2.2.2 is based on the distinction
between shifted lattice points $\widetilde{m}\in\mathcal{L}_{\kappa}^{\left(+\right)}\left(\infty\right)$
and $\widetilde{m}\in\mathcal{L}_{\kappa}^{\left(-\right)}\left(\infty\right)$,
respectively, whereas §2.1.1 and §2.1.2 were organized around two
distinct components of the preliminary analysis.} Given that the derivations closely parallel those of the preceding
subsection, we will omit redundancies of exposition whenever possible
by careful referral to analogous portions of the earlier presentation.
\end{singlespace}

\begin{singlespace}
As noted in the Introduction, the only shift values $\kappa\in\left(0,1\right)\cap\mathbb{Q}$
for which it is known that $\Gamma\left(\kappa\right)$ is transcendental
are those in the set $\mathcal{K}^{*}=\left\{ 1/6,1/4,1/3,1/2,2/3,3/4,5/6\right\} $.
For $\kappa=1/2$, transcendence is immediately apparent because $\Gamma\left(1/2\right)=\sqrt{\pi}$,
whereas for $\kappa=1/4,1/3$ it follows from the work of Chudnovsky
{[}4{]}, who showed that both $\left\{ \Gamma\left(1/4\right),\pi\right\} $
and $\left\{ \Gamma\left(1/3\right),\pi\right\} $ are algebraically
independent pairs. The transcendence of $\Gamma\left(1/6\right)$
can be demonstrated by applying Legendre's duplication formula ($\Gamma\left(2z\right)=\left(2^{2z-1}/\sqrt{\pi}\right)\Gamma\left(z\right)\Gamma\left(z+1/2\right)$)
and Euler's reflection formula ($\Gamma\left(z\right)\Gamma\left(1-z\right)=\pi/\sin\left(\pi z\right)$)
to Chudnovsky's result for $\Gamma\left(1/3\right)$, and the complementary
cases of $\Gamma\left(2/3\right)$, $\Gamma\left(3/4\right)$, and
$\Gamma\left(5/6\right)$ then follow from Euler's reflection formula
in conjunction with both of Chudnovsky's results (see, e.g., Murty
and Weatherby {[}5{]}). Despite the rather small size of $\mathcal{K}^{*}$,
it is generally believed that $\Gamma\left(q\right)$ is transcendental
for all $q\in\left(0,1\right)\cap\mathbb{Q}$. Thus, our analysis
is broadly applicable conditional on this common view.\medskip{}

\end{singlespace}

\begin{singlespace}
\noindent\textbf{\emph{2.2.1\enskip{}The Case of $\boldsymbol{\widetilde{m}\in\mathcal{L}_{\kappa}^{\left(+\right)}\left(\infty\right)}$}}\medskip{}

\noindent For $m\in\mathbb{Z}_{\ge0}$, the functional equation gives
\[
\Gamma\left(m+\kappa+t\right)=\Gamma\left(\kappa+t\right)\prod_{u=0}^{m-1}\left(u+\kappa+t\right)
\]
\[
=\dfrac{\Gamma\left(m+\kappa\right)}{\Gamma\left(\kappa\right)}\Gamma\left(\kappa+t\right)\prod_{u=0}^{m-1}\left(1+\frac{t}{u+\kappa}\right)
\]
\begin{equation}
=\dfrac{\Gamma\left(m+\kappa\right)}{\Gamma\left(\kappa\right)}\Gamma\left(\kappa+t\right)\sum_{v=0}^{m}e_{v}\left(\dfrac{1}{\kappa},\frac{1}{1+\kappa},\dots,\frac{1}{m-1+\kappa}\right)t^{v},
\end{equation}
where ${\textstyle \prod_{u=0}^{m-1}\left(u+\kappa+t\right)\coloneqq1}$
for $m=0$ and $e_{v}\left(x_{1},x_{2},\dots,x_{k}\right)$ is defined
as in §2.1.1 (with the argument of $e_{v}\left(1/\kappa,1/\left(1+\kappa\right),\dots,1/\left(m-1+\kappa\right)\right)$
treated as $\emptyset$ for $m=0$). Differentiating (9) $n$ times
with respect to $t$ and then evaluating at $t=0$ gives
\[
\Gamma^{\left(n\right)}\left(m+\kappa\right)=\dfrac{\Gamma\left(m+\kappa\right)}{\Gamma\left(\kappa\right)}\sum_{\ell=0}^{n}\dfrac{n!}{\ell!}e_{n-\ell}\left(\dfrac{1}{\kappa},\frac{1}{1+\kappa},\dots,\frac{1}{m-1+\kappa}\right)\Gamma^{\left(\ell\right)}\left(\kappa\right)
\]
\begin{equation}
=\sum_{\ell=0}^{n}T_{n,\ell}^{\left(+\right)}\left(m+\kappa\right)\Gamma^{\left(\ell\right)}\left(\kappa\right),
\end{equation}
where
\[
T_{n,\ell}^{\left(+\right)}\left(m+\kappa\right)=\dfrac{\Gamma\left(m+\kappa\right)}{\Gamma\left(\kappa\right)}\frac{n!}{\ell!}e_{n-\ell}\left(\dfrac{1}{\kappa},\frac{1}{1+\kappa},\dots,\frac{1}{m-1+\kappa}\right)\in\mathbb{Q}
\]
and $\Gamma^{\left(0\right)}\left(\kappa\right)=\Gamma\left(\kappa\right)$.
\end{singlespace}

\begin{singlespace}
The identity in (10), which holds for all $n\in\mathbb{Z}_{\geq0}$
and $m\in\mathbb{Z}_{\geq0}$, allows us to construct a system of
$n+1\in\mathbb{Z}_{\geq2}$ equations of this form, based on a set
of arbitrary $\widetilde{m}$ values, $\mathcal{M}_{\kappa}^{\left(+\right)}=\left\{ m_{1}+\kappa,m_{2}+\kappa,\ldots,m_{n+1}+\kappa\right\} $
with $m_{r}\in\mathbb{Z}_{\geq0}$ for all $r\in\left\{ 1,2,\ldots,n+1\right\} $
and $m_{1}<m_{2}<\ldots<m_{n+1}$. Specifically,
\begin{equation}
\left[\begin{array}{c}
\Gamma^{\left(n\right)}\left(m_{1}+\kappa\right)\\
\Gamma^{\left(n\right)}\left(m_{2}+\kappa\right)\\
\vdots\\
\Gamma^{\left(n\right)}\left(m_{n+1}+\kappa\right)
\end{array}\right]=\mathbf{T}_{n+1}^{\left(+\right)}\left(\mathcal{M}_{\kappa}^{\left(+\right)}\right)\left[\begin{array}{c}
\Gamma^{\left(0\right)}\left(\kappa\right)\\
\Gamma^{\left(1\right)}\left(\kappa\right)\\
\vdots\\
\Gamma^{\left(n\right)}\left(\kappa\right)
\end{array}\right],
\end{equation}
with $\left(n+1\right)\times\left(n+1\right)$ coefficient matrix
\[
\mathbf{T}_{n+1}^{\left(+\right)}\left(\mathcal{M}_{\kappa}^{\left(+\right)}\right)=\left[\begin{array}{cccc}
T_{n,0}^{\left(+\right)}\left(m_{1}+\kappa\right) & T_{n,1}^{\left(+\right)}\left(m_{1}+\kappa\right) & \cdots & T_{n,n}^{\left(+\right)}\left(m_{1}+\kappa\right)\\
T_{n,0}^{\left(+\right)}\left(m_{2}+\kappa\right) & T_{n,1}^{\left(+\right)}\left(m_{2}+\kappa\right) & \cdots & T_{n,n}^{\left(+\right)}\left(m_{2}+\kappa\right)\\
\vdots & \vdots & \ddots & \vdots\\
T_{n,0}^{\left(+\right)}\left(m_{n+1}+\kappa\right) & T_{n,1}^{\left(+\right)}\left(m_{n+1}+\kappa\right) & \cdots & T_{n,n}^{\left(+\right)}\left(m_{n+1}+\kappa\right)
\end{array}\right].
\]

To show that $\mathbf{T}_{n+1}^{\left(+\right)}\left(\mathcal{M}_{\kappa}^{\left(+\right)}\right)$
is nonsingular, we proceed as in the proofs of Lemmas 1 and 2 after
replacing the dimension $n$ by $n+1$, the factor $\left(m_{r}-1\right)!$
by $\Gamma\left(m_{r}+\kappa\right)/\Gamma\left(\kappa\right)$, and
the argument vector $\left(1,1/2,\dots,1/\left(m_{r}-1\right)\right)$
by $\left(1/\kappa,1/\left(1+\kappa\right),\dots,1/\left(m_{r}-1+\kappa\right)\right)$.
Specifically, we multiply the $\left(\ell+1\right)^{\textrm{th}}$
column of $\mathbf{T}_{n+1}^{\left(+\right)}\left(\mathcal{M}_{\kappa}^{\left(+\right)}\right)$
by the non-zero rational factor $\ell!/n!$ and then reverse the column
order, transforming the matrix into a row-scaled version of the matrix
$\mathbf{E}_{n+1}\left(\mathcal{M}^{\prime}\right)$ (analogous to
$\mathbf{E}_{n}\left(\mathcal{M}^{\prime}\right)$ of Lemma 1) with
$m_{r}^{\prime}=m_{r}$ and $x_{s}=1/\left(s-1+\kappa\right)$ for
$s\in\left\{ 1,2,\ldots,m_{n+1}\right\} $. Since this transformation
does not affect nonsingularity, it follows from subsequent arguments
in the proofs of Lemmas 1 and 2 that
\[
\det\left(\mathbf{T}_{n+1}^{\left(+\right)}\left(\mathcal{M}_{\kappa}^{\left(+\right)}\right)\right)\neq0,
\]
so that $\mathbf{T}_{n+1}^{\left(+\right)}\left(\mathcal{M}_{\kappa}^{\left(+\right)}\right)$
is nonsingular. Given that all entries of the matrix lie in $\mathbb{Q}$,
we know that all entries of its inverse lie in $\mathbb{Q}$ as well.

\medskip{}

\end{singlespace}

\begin{singlespace}
\noindent\textbf{\emph{2.2.2\enskip{}The Case of $\boldsymbol{\widetilde{m}\in\mathcal{L}_{\kappa}^{\left(-\right)}\left(\infty\right)}$}}\medskip{}

\noindent For $m\in\mathbb{Z}_{\ge0}$, the functional equation gives
\[
\Gamma\left(-m+\kappa+t\right)=\dfrac{\Gamma\left(\kappa+t\right)}{\prod_{u=1}^{m}\left(-u+\kappa+t\right)}
\]
\[
=\dfrac{\Gamma\left(\kappa+t\right)}{\dfrac{\Gamma\left(\kappa\right)}{\Gamma\left(-m+\kappa\right)}\prod_{u=1}^{m}\left(1-\dfrac{t}{u-\kappa}\right)}
\]
\[
=\dfrac{\Gamma\left(-m+\kappa\right)}{\Gamma\left(\kappa\right)}\Gamma\left(\kappa+t\right)\prod_{u=1}^{m}\left(1-\dfrac{t}{u-\kappa}\right)^{-1}
\]
\begin{equation}
=\dfrac{\Gamma\left(-m+\kappa\right)}{\Gamma\left(\kappa\right)}\Gamma\left(\kappa+t\right)\sum_{v=0}^{\infty}h_{v}\left(\dfrac{1}{1-\kappa},\frac{1}{2-\kappa},\dots,\frac{1}{m-\kappa}\right)t^{v},
\end{equation}
where ${\textstyle \prod_{u=1}^{m}\left(-u+\kappa+t\right)\coloneqq1}$
for $m=0$ and $h_{v}\left(x_{1},x_{2},\dots,x_{k}\right)$ denotes
the $v^{\textrm{th}}$ complete homogeneous symmetric polynomial with
$h_{0}\left(x_{1},x_{2},\dots,x_{k}\right)\coloneqq1$, $h_{0}\left(\emptyset\right)\coloneqq1$,
$h_{v}\left(x_{1},x_{2},\dots,x_{k}\right)\coloneqq0$ for $v<0$,
and $h_{v}\left(\emptyset\right)\coloneqq0$ for $v>0$ (in particular,
the argument of $h_{v}\left(1/\left(1-\kappa\right),1/\left(2-\kappa\right),\dots,1/\left(m-\kappa\right)\right)$
is treated as $\emptyset$ for $m=0$). Differentiating (12) $n$
times with respect to $t$ at $t=0$ gives
\[
\Gamma^{\left(n\right)}\left(-m+\kappa\right)=\dfrac{\Gamma\left(-m+\kappa\right)}{\Gamma\left(\kappa\right)}\sum_{\ell=0}^{n}\dfrac{n!}{\ell!}h_{n-\ell}\left(\dfrac{1}{1-\kappa},\frac{1}{2-\kappa},\dots,\frac{1}{m-\kappa}\right)\Gamma^{\left(\ell\right)}\left(\kappa\right)
\]
\begin{equation}
=\sum_{\ell=0}^{n}T_{n,\ell}^{\left(-\right)}\left(-m+\kappa\right)\Gamma^{\left(\ell\right)}\left(\kappa\right),
\end{equation}
where
\[
T_{n,\ell}^{\left(-\right)}\left(-m+\kappa\right)=\dfrac{\Gamma\left(-m+\kappa\right)}{\Gamma\left(\kappa\right)}\frac{n!}{\ell!}h_{n-\ell}\left(\dfrac{1}{1-\kappa},\frac{1}{2-\kappa},\dots,\frac{1}{m-\kappa}\right)\in\mathbb{Q}.
\]

\end{singlespace}

The identity in (13), which holds for all $n\in\mathbb{Z}_{\geq0}$
and $m\in\mathbb{Z}_{\geq0}$, allows us to construct a system of
$n+1\in\mathbb{Z}_{\geq2}$ equations of the same form, based on a
set of arbitrary $\widetilde{m}$ values, $\mathcal{M}_{\kappa}^{\left(-\right)}=\left\{ -m_{1}+\kappa,-m_{2}+\kappa,\ldots,-m_{n+1}+\kappa\right\} $
with $m_{r}\in\mathbb{Z}_{\geq0}$ for all $r\in\left\{ 1,2,\ldots,n+1\right\} $
and $m_{1}<m_{2}<\ldots<m_{n+1}$. Specifically,
\begin{equation}
\left[\begin{array}{c}
\Gamma^{\left(n\right)}\left(-m_{1}+\kappa\right)\\
\Gamma^{\left(n\right)}\left(-m_{2}+\kappa\right)\\
\vdots\\
\Gamma^{\left(n\right)}\left(-m_{n+1}+\kappa\right)
\end{array}\right]=\mathbf{T}_{n+1}^{\left(-\right)}\left(\mathcal{M}_{\kappa}^{\left(-\right)}\right)\left[\begin{array}{c}
\Gamma^{\left(0\right)}\left(\kappa\right)\\
\Gamma^{\left(1\right)}\left(\kappa\right)\\
\vdots\\
\Gamma^{\left(n\right)}\left(\kappa\right)
\end{array}\right],
\end{equation}
with $\left(n+1\right)\times\left(n+1\right)$ coefficient matrix
\[
\mathbf{T}_{n+1}^{\left(-\right)}\left(\mathcal{M}_{\kappa}^{\left(-\right)}\right)=\left[\begin{array}{cccc}
T_{n,0}^{\left(-\right)}\left(-m_{1}+\kappa\right) & T_{n,1}^{\left(-\right)}\left(-m_{1}+\kappa\right) & \cdots & T_{n,n}^{\left(-\right)}\left(-m_{1}+\kappa\right)\\
T_{n,0}^{\left(-\right)}\left(-m_{2}+\kappa\right) & T_{n,1}^{\left(-\right)}\left(-m_{2}+\kappa\right) & \cdots & T_{n,n}^{\left(-\right)}\left(-m_{2}+\kappa\right)\\
\vdots & \vdots & \ddots & \vdots\\
T_{n,0}^{\left(-\right)}\left(-m_{n+1}+\kappa\right) & T_{n,1}^{\left(-\right)}\left(-m_{n+1}+\kappa\right) & \cdots & T_{n,n}^{\left(-\right)}\left(-m_{n+1}+\kappa\right)
\end{array}\right].
\]

\begin{singlespace}
As with $\mathbf{T}_{n+1}^{\left(+\right)}\left(\mathcal{M}_{\kappa}^{\left(+\right)}\right)$
in §2.2.1, it can be shown that $\mathbf{T}_{n+1}^{\left(-\right)}\left(\mathcal{M}_{\kappa}^{\left(-\right)}\right)$
is nonsingular via arguments parallel to those underlying Lemmas 1
and 2. However, in this case we make a few additional adjustments
to account for the use of complete homogeneous symmetric polynomials
($h_{v}\left(\cdot\right)$) rather than elementary symmetric polynomials
($e_{v}\left(\cdot\right)$). 

After replacing the dimension $n$ by $n+1$, the factor $\left(m_{r}-1\right)!$
by $\Gamma\left(-m_{r}+\kappa\right)/\Gamma\left(\kappa\right)$,
and the argument vector $\left(1,1/2,\dots,1/\left(m_{r}-1\right)\right)$
by $\left(1/\left(1-\kappa\right),1/\left(2-\kappa\right),\dots,1/\left(m_{r}-\kappa\right)\right)$,
we multiply the $\left(\ell+1\right)^{\textrm{th}}$ column of $\mathbf{T}_{n+1}^{\left(-\right)}\left(\mathcal{M}_{\kappa}^{\left(-\right)}\right)$
by the non-zero rational factor $\ell!/n!$ and then reverse the column
order. This transforms $\mathbf{T}_{n+1}^{\left(-\right)}\left(\mathcal{M}_{\kappa}^{\left(-\right)}\right)$
into a row-scaled version of the $\left(n+1\right)\times\left(n+1\right)$
matrix 
\[
\mathbf{H}_{n+1}\left(\mathcal{M}^{\prime}\right)=\left[\begin{array}{cccc}
1 & h_{1}\left(\boldsymbol{x}^{\left(m_{1}^{\prime}\right)}\right) & \cdots & h_{n}\left(\boldsymbol{x}^{\left(m_{1}^{\prime}\right)}\right)\\
1 & h_{1}\left(\boldsymbol{x}^{\left(m_{2}^{\prime}\right)}\right) & \cdots & h_{n}\left(\boldsymbol{x}^{\left(m_{2}^{\prime}\right)}\right)\\
\vdots & \vdots & \ddots & \vdots\\
1 & h_{1}\left(\boldsymbol{x}^{\left(m_{n+1}^{\prime}\right)}\right) & \cdots & h_{n}\left(\boldsymbol{x}^{\left(m_{n+1}^{\prime}\right)}\right)
\end{array}\right]
\]
(analogous to $\mathbf{E}_{n}\left(\mathcal{M}^{\prime}\right)$ of
Lemma 1), with $m_{r}^{\prime}=m_{r}$, $\boldsymbol{x}^{\left(m_{r}^{\prime}\right)}=\left[x_{1},x_{2},\ldots,x_{m_{r}^{\prime}}\right]$,
and $x_{s}=1/\left(s-\kappa\right)$ for $s\in\left\{ 1,2,\ldots,m_{n+1}\right\} $.
Since this transformation does not affect nonsingularity, it then
suffices to show that
\begin{equation}
\det\left(\mathbf{H}_{n+1}\left(\mathcal{M}^{\prime}\right)\right)>0.
\end{equation}

\end{singlespace}

\begin{singlespace}
To confirm that this inequality holds, we repeat the induction argument
used in the proof of Lemma 1, beginning with $n=0$, for which
\[
\mathbf{H}_{1}\left(\mathcal{M}^{\prime}\right)=\left[1\right]\Longrightarrow\det\left(\mathbf{H}_{1}\left(\mathcal{M}^{\prime}\right)\right)=1.
\]
Assuming (15) is true for arbitrary $n$, we proceed to construct
a row-difference matrix analogous to that in Lemma 1, and then expand
by minors to obtain a new $n\times n$ matrix $\mathbf{D}_{n}\left(\mathcal{M}^{\prime}\right)$.
Replacing (5) by
\[
h_{c}\left(\boldsymbol{x}^{\left(m_{r}^{\prime}\right)}\right)=\begin{cases}
h_{c}\left(\boldsymbol{x}^{\left(m_{r}^{\prime}-1\right)}\right)+x_{m_{r}^{\prime}}h_{c-1}\left(\boldsymbol{x}^{\left(m_{r}^{\prime}\right)}\right), & m_{r}^{\prime}\geq1\\
h_{c}\left(\emptyset\right), & m_{r}^{\prime}=0
\end{cases}
\]
to derive
\[
h_{c}\left(\boldsymbol{x}^{\left(j\right)}\right)-h_{c}\left(\boldsymbol{x}^{\left(j-1\right)}\right)=x_{j}h_{c-1}\left(\boldsymbol{x}^{\left(j\right)}\right),
\]
we find that the ($r,c$) entry of $\mathbf{D}_{n}\left(\mathcal{M}^{\prime}\right)$
is given by
\[
D_{r,c}=\sum_{j=m_{r}^{\prime}+1}^{m_{r+1}^{\prime}}x_{j}h_{c-1}\left(\boldsymbol{x}^{\left(j\right)}\right).
\]
This allows us to write $\mathbf{D}_{n}\left(\mathcal{M}^{\prime}\right)=\mathbf{A}\mathbf{B}$,
where the new $n\times m_{n+1}^{\prime}$ matrix $\mathbf{A}$ is
defined by its generic ($r,j$) entry
\[
A_{r,j}=\begin{cases}
x_{j}, & m_{r}^{\prime}<j\leq m_{r+1}^{\prime}\\
0, & \text{otherwise}
\end{cases},
\]
the new $m_{n+1}^{\prime}\times n$ matrix $\mathbf{B}$ is defined
by its generic ($j,c$) entry
\[
B_{j,c}=h_{c-1}\left(\boldsymbol{x}^{\left(j\right)}\right),
\]
and the Cauchy-Binet formula gives
\[
\det\left(\mathbf{D}_{n}\left(\mathcal{M}^{\prime}\right)\right)=\sum_{\substack{\mathcal{S}\subseteq\left\{ 1,2,\ldots,m_{n+1}^{\prime}\right\} ,\\
|\mathcal{S}|=n
}
}\det\left(\mathbf{A}_{*,\mathcal{S}}\right)\det\left(\mathbf{B}_{\mathcal{S},*}\right).
\]

\end{singlespace}

As in the proof of Lemma 1, the supports of the rows of $\mathbf{A}$
are the disjoint intervals $\left(m_{r}^{\prime},m_{r+1}^{\prime}\right]$.
This means that $\det\left(\mathbf{A}_{*,\mathcal{S}}\right)=0$ unless
$\mathcal{S}=\left\{ s_{1}<s_{2}<\cdots<s_{n}\right\} $ with $m_{r}^{\prime}<s_{r}\leq m_{r+1}^{\prime}$
for all $r$, in which case $\mathbf{A}_{*,\mathcal{S}}$ is diagonal
with diagonal entries $x_{s_{1}},x_{s_{2}},\dots,x_{s_{n}}$, implying
$\det\left(\mathbf{A}_{*,\mathcal{S}}\right)=\prod_{r=1}^{n}x_{s_{r}}>0$.
Moreover, for any such $\mathcal{S}$, the matrix $\mathbf{B}_{\mathcal{S},*}$
has entries 
\[
B_{r^{\prime},c^{\prime}}^{\left(\mathcal{S},*\right)}=h_{c^{\prime}-1}\left(\boldsymbol{x}^{\left(s_{r^{\prime}}\right)}\right),
\]
with $1\le s_{1}<s_{2}<\cdots<s_{n}$, giving $\mathbf{B}_{\mathcal{S},*}$
the same form as $\mathbf{H}_{n}\left(\mathcal{M}^{\prime}\right)$,
so that $\det\left(\mathbf{B}_{\mathcal{S},*}\right)>0$ by the induction
hypothesis. Applying subsequent arguments from the proofs of Lemmas
1 and 2, we find that
\[
\det\left(\mathbf{D}_{n}\left(\mathcal{M}^{\prime}\right)\right)>0
\]
\[
\Longrightarrow\det\left(\mathbf{H}_{n+1}\left(\mathcal{M}^{\prime}\right)\right)>0
\]
\[
\Longrightarrow\det\left(\mathbf{T}_{n+1}^{\left(-\right)}\left(\mathcal{M}_{\kappa}^{\left(-\right)}\right)\right)\neq0,
\]
and thus that $\mathbf{T}_{n+1}^{\left(-\right)}\left(\mathcal{M}_{\kappa}^{\left(-\right)}\right)$
is nonsingular. Since all entries of $\mathbf{T}_{n+1}^{\left(-\right)}\left(\mathcal{M}_{\kappa}^{\left(-\right)}\right)$
lie in $\mathbb{Q}$, it follows that all entries of its inverse lie
in $\mathbb{Q}$ as well.

\begin{singlespace}
\medskip{}

\end{singlespace}

\begin{singlespace}
\noindent\textbf{\emph{2.2.3\enskip{}Upper Bound for Number of Algebraic
Values}}\medskip{}

\end{singlespace}

\begin{singlespace}
\noindent Given that both coefficient matrices $\mathbf{T}_{n+1}^{\left(+\right)}\left(\mathcal{M}_{\kappa}^{\left(+\right)}\right)$
and $\mathbf{T}_{n+1}^{\left(-\right)}\left(\mathcal{M}_{\kappa}^{\left(-\right)}\right)$
are invertible, we now use the linear systems in (11) and (14) (for
$\widetilde{m}\in\mathcal{L}_{\kappa}^{\left(+\right)}\left(\infty\right)$
and $\widetilde{m}\in\mathcal{L}_{\kappa}^{\left(-\right)}\left(\infty\right)$,
respectively) to provide upper bounds for the number of algebraic
values $\Gamma^{\left(n\right)}\left(\widetilde{m}\right)$ for fixed
$n\in\mathbb{Z}_{\geq1}$.
\end{singlespace}
\begin{thm}
For $n\in\mathbb{Z}_{\geq1}$ and $\kappa\in\mathcal{K}$:

\begin{singlespace}
\noindent
\[
\mathit{(i)}\;\#\left(\widetilde{m}\in\mathcal{L}_{\kappa}^{\left(+\right)}\left(\infty\right):\Gamma^{\left(n\right)}\left(\widetilde{m}\right)\textrm{ is algebraic}\right)\le n
\]
 and

\noindent
\[
\mathit{(ii)}\;\#\left(\widetilde{m}\in\mathcal{L}_{\kappa}^{\left(-\right)}\left(\infty\right):\Gamma^{\left(n\right)}\left(\widetilde{m}\right)\textrm{ is algebraic}\right)\le n.
\]
\pagebreak{}
\end{singlespace}
\end{thm}
\begin{proof}
\begin{singlespace}
\phantom{}
\end{singlespace}

\begin{singlespace}
\medskip{}
\noindent (i) Fixing $n$ and $\kappa$, consider a set $\mathcal{M}_{\kappa}^{\left(+\right)}=\left\{ m_{1}+\kappa,m_{2}+\kappa,\ldots,m_{n+1}+\kappa\right\} $
with $m_{r}\in\mathbb{Z}_{\geq0}$ for all $r\in\left\{ 1,2,\ldots,n+1\right\} $
and $m_{1}<m_{2}<\ldots<m_{n+1}$, and assume (for purposes of contradiction)
that $\Gamma^{\left(n\right)}\left(m_{r}+\kappa\right)\in\overline{\mathbb{Q}}$
for all $r\in\left\{ 1,2,\ldots,n+1\right\} $. Knowing that $\mathbf{T}_{n+1}^{\left(+\right)}\left(\mathcal{M}_{\kappa}^{\left(+\right)}\right)$
is invertible, we can solve for the column vector $\left[\Gamma^{\left(0\right)}\left(\kappa\right),\Gamma^{\left(1\right)}\left(\kappa\right),\ldots,\Gamma^{\left(n\right)}\left(\kappa\right)\right]^{\mathrm{T}}$
in (11) as
\begin{equation}
\left[\begin{array}{c}
\Gamma^{\left(0\right)}\left(\kappa\right)\\
\Gamma^{\left(1\right)}\left(\kappa\right)\\
\vdots\\
\Gamma^{\left(n\right)}\left(\kappa\right)
\end{array}\right]=\left(\mathbf{T}_{n+1}^{\left(+\right)}\left(\mathcal{M}_{\kappa}^{\left(+\right)}\right)\right)^{-1}\left[\begin{array}{c}
\Gamma^{\left(n\right)}\left(m_{1}+\kappa\right)\\
\Gamma^{\left(n\right)}\left(m_{2}+\kappa\right)\\
\vdots\\
\Gamma^{\left(n\right)}\left(m_{n+1}+\kappa\right)
\end{array}\right],
\end{equation}
where all elements of the column vector on the right-hand side of
(16) are algebraic under the contradiction hypothesis.
\end{singlespace}

Since the inverse matrix $\left(\mathbf{T}_{n+1}^{\left(+\right)}\left(\mathcal{M}_{\kappa}^{\left(+\right)}\right)\right)^{-1}$
contains only rational entries, it follows that all elements of\linebreak{}
$\left[\Gamma^{\left(0\right)}\left(\kappa\right),\Gamma^{\left(1\right)}\left(\kappa\right),\ldots,\Gamma^{\left(n\right)}\left(\kappa\right)\right]^{\mathrm{T}}$
must be algebraic. However, we know that $\Gamma^{\left(0\right)}\left(\kappa\right)$
is transcendental, forcing a contradiction that precludes $\Gamma^{\left(n\right)}\left(m_{r}+\kappa\right)\in\overline{\mathbb{Q}}$
for all $r\in\left\{ 1,2,\ldots,n+1\right\} $. Therefore, at most
$n$ of the elements in $\left[\Gamma^{\left(n\right)}\left(m_{1}+\kappa\right),\Gamma^{\left(n\right)}\left(m_{2}+\kappa\right),\ldots,\Gamma^{\left(n\right)}\left(m_{n+1}+\kappa\right)\right]^{\mathrm{T}}$
can be algebraic. Furthermore, given that $\mathcal{M}_{\kappa}^{\left(+\right)}$
was constructed of arbitrary $m\in\mathbb{Z}_{\geq0}$, it is impossible
to assemble any set of algebraic $\Gamma^{\left(n\right)}\left(m+\kappa\right)$
with more than $n$ distinct $m$.

\begin{singlespace}
\noindent (ii) The proof is entirely analogous to that of part (i),
with $\mathcal{M}_{\kappa}^{\left(+\right)}$ replaced by $\mathcal{M}_{\kappa}^{\left(-\right)}=$\linebreak{}
$\left\{ -m_{1}+\kappa,-m_{2}+\kappa,\ldots,-m_{n+1}+\kappa\right\} $
(where $m_{r}\in\mathbb{Z}_{\geq0}$ for all $r\in\left\{ 1,2,\ldots,n+1\right\} $
and $m_{1}<m_{2}<\ldots<m_{n+1}$) and $\left(\mathbf{T}_{n+1}^{\left(+\right)}\left(\mathcal{M}_{\kappa}^{\left(+\right)}\right)\right)^{-1}$
replaced by $\left(\mathbf{T}_{n+1}^{\left(-\right)}\left(\mathcal{M}_{\kappa}^{\left(-\right)}\right)\right)^{-1}$.
\end{singlespace}
\end{proof}
\begin{singlespace}
This result, like Theorem 1, is rather strong because it demonstrates
the transcendence of $\Gamma^{\left(n\right)}\left(\widetilde{m}\right)$
for ``all but a finite number of $\widetilde{m}$'' for each $n\in\mathbb{Z}_{\geq1}$.
In particular, for the lowest value of $n$ addressed by Theorem 2
(i.e., $n=1$), we know that $\Gamma^{\left(1\right)}\left(\widetilde{m}\right)$
is transcendental for ``all but at most one $\widetilde{m}$'' for
each one-sided shifted lattice (i.e., $\widetilde{m}\geq\kappa$ or
$\widetilde{m}\leq\kappa$).
\end{singlespace}

\begin{singlespace}
\medskip{}

\end{singlespace}

\begin{singlespace}
\noindent\textbf{\emph{2.2.4\enskip{}Lower Bound for Transcendental
Densities}}\medskip{}

\end{singlespace}

\begin{singlespace}
\noindent For fixed $n\in\mathbb{Z}_{\geq1}$ and $\kappa\in\mathcal{K}$,
let
\[
\delta_{n,\kappa}^{\left(+\right)}\left(M\right)=\dfrac{\#\left\{ \widetilde{m}\in\mathcal{L}_{\kappa}^{\left(+\right)}\left(M\right):\Gamma^{\left(n\right)}\left(\widetilde{m}\right)\textrm{ is transcendental}\right\} }{M+1},\quad M\in\mathbb{Z}_{\geq0}
\]
and
\[
\delta_{n,\kappa}^{\left(-\right)}\left(M\right)=\dfrac{\#\left\{ \widetilde{m}\in\mathcal{L}_{\kappa}^{\left(-\right)}\left(M\right):\Gamma^{\left(n\right)}\left(\widetilde{m}\right)\textrm{ is transcendental}\right\} }{M+1},\quad M\in\mathbb{Z}_{\geq0}
\]
denote the densities of the sets of transcendental $\Gamma^{\left(n\right)}\left(\widetilde{m}\right)$
among $\widetilde{m}\in\mathcal{L}_{\kappa}^{\left(+\right)}\left(M\right)$
and $\widetilde{m}\in\mathcal{L}_{\kappa}^{\left(-\right)}\left(M\right)$,
respectively. The following corollary to Theorem 2 gives a lower bound
for these densities.
\end{singlespace}
\begin{cor}
For $n\in\mathbb{Z}_{\geq1}$ and $\kappa\in\mathcal{K}$:
\begin{equation}
\mathit{(i)}\;\delta_{n,\kappa}^{\left(+\right)}\left(M\right)\geq\widetilde{\beta}_{n}\left(M\right)=1-\dfrac{\min\left\{ n,M+1\right\} }{M+1}
\end{equation}
and
\begin{equation}
\mathit{(ii)}\;\delta_{n,\kappa}^{\left(-\right)}\left(M\right)\geq\widetilde{\beta}_{n}\left(M\right)=1-\dfrac{\min\left\{ n,M+1\right\} }{M+1},
\end{equation}
where
\[
\widetilde{\beta}_{n}\left(M\right)\asymp1-\dfrac{n}{M}\;as\;M\rightarrow\infty.
\]
\pagebreak{} 
\end{cor}
\begin{proof}
\begin{singlespace}
\phantom{}
\end{singlespace}

\begin{singlespace}
\medskip{}
\noindent (i), (ii) For values of $M+1>n$, Theorem 2 implies that the density
of algebraic $\Gamma^{\left(n\right)}\left(\widetilde{m}\right)$
among either $\widetilde{m}\in\mathcal{L}_{\kappa}^{\left(+\right)}\left(M\right)$
or $\widetilde{m}\in\mathcal{L}_{\kappa}^{\left(-\right)}\left(M\right)$
is (weakly) bounded above by $n/\left(M+1\right)$. On the other hand,
for $M+1\leq n$, Theorem 2 does not preclude all $\Gamma^{\left(n\right)}\left(\widetilde{m}\right)$
from being algebraic, resulting in the (weak) upper bound of $\left(M+1\right)/\left(M+1\right)=1$.
To obtain a lower bound for the transcendental density, we take the
complements of these upper bounds, giving
\[
\delta_{n,\kappa}^{\left(+\right)}\left(M\right)\geq\widetilde{\beta}_{n}\left(M\right)=\begin{cases}
0/\left(M+1\right), & M+1\leq n\\
1-n/\left(M+1\right), & M+1>n
\end{cases}
\]
and
\[
\delta_{n,\kappa}^{\left(-\right)}\left(M\right)\geq\widetilde{\beta}_{n}\left(M\right)=\begin{cases}
0/\left(M+1\right), & M+1\leq n\\
1-n/\left(M+1\right), & M+1>n
\end{cases},
\]
from which (17) and (18) follow.
\end{singlespace}
\end{proof}
\medskip{}

\section{Lower Bounds with Variable $n$, Variable $m$}

\begin{singlespace}
\noindent In the present section, we extend the analysis of Section
2 to evaluate the bivariate densities of both transcendental $\Gamma^{\left(n\right)}\left(m\right)$
among $\left(n,m\right)\in\left\{ 2,3,\ldots,N\right\} \times\mathcal{L}\left(M\right)$
and transcendental $\Gamma^{\left(n\right)}\left(\widetilde{m}\right)$
among $\left(n,\widetilde{m}\right)\in\left\{ 1,2,\ldots,N\right\} \times\mathcal{L}_{\kappa}^{\left(+\right)}\left(M\right)$
and $\left(n,\widetilde{m}\right)\in\left\{ 1,2,\ldots,N\right\} \times\mathcal{L}_{\kappa}^{\left(-\right)}\left(M\right)$.
The separate cases of positive lattice points and rationally shifted
lattice points are addressed in Subsections 3.1 and 3.2, respectively.
\end{singlespace}

\subsection{Positive Lattice Points}

\begin{singlespace}
\phantom{}

\medskip{}

\end{singlespace}

\begin{singlespace}
\noindent For variable $n\in\mathbb{Z}_{\geq2}$ and positive lattice
points $m\in\mathcal{L}\left(\infty\right)$, define the bivariate
density
\[
\delta\left(N,M\right)=\dfrac{\#\left(\left(n,m\right)\in\left\{ 2,3,\ldots,N\right\} \times\mathcal{L}\left(M\right):\Gamma^{\left(n\right)}\left(m\right)\textrm{ is transcendental}\right)}{\left(N-1\right)M}.
\]
The following result provides a lower bound for $\delta\left(N,M\right)$
by leveraging Theorem 1 to construct an upper bound for the number
of algebraic values among pairs $\left(n,m\right)\in\left\{ 2,3,\ldots,N\right\} \times\mathcal{L}\left(M\right)$.
\end{singlespace}
\begin{thm}
For $N\in\mathbb{Z}_{\geq2}$ and $M\in\mathbb{Z}_{\geq1}$,
\[
\delta\left(N,M\right)\ge\beta\left(N,M\right)=\begin{cases}
\left(M-1\right)/\left[2\left(N-1\right)\right], & M\leq N-1\\
1-N/\left(2M\right), & M>N-1
\end{cases},
\]
where
\[
\beta\left(N,M\right)\asymp\begin{cases}
M/\left(2N\right), & M/N\leq1\\
1-N/\left(2M\right), & M/N>1
\end{cases}
\]
as $N,M\rightarrow\infty$.
\end{thm}
\begin{proof}
\begin{singlespace}
\phantom{}
\end{singlespace}

\begin{singlespace}
\medskip{}
\end{singlespace}

\begin{singlespace}
\noindent For a given $n\in\mathbb{Z}_{\geq2}$, we know from Theorem
1 that the $M$ integers $m\in\left\{ 1,2,\ldots,M\right\} $ yield
at most $\min\left\{ n-1,M\right\} $ algebraic values of $\Gamma^{\left(n\right)}\left(m\right)$,
and therefore at least $M-\min\left\{ n-1,M\right\} $ transcendental
values. Summing the latter expression over $n\in\left\{ 2,3,\dots,N\right\} $
then gives a (weak) lower bound for the total number of transcendental
values,
\[
\left(N-1\right)M\delta\left(N,M\right)\ge\sum_{n=2}^{N}\left(M-\min\left\{ n-1,M\right\} \right)=\left(N-1\right)M-\sum_{n=2}^{N}\min\left\{ n-1,M\right\} 
\]
\[
\Longrightarrow\delta\left(N,M\right)\ge\beta\left(N,M\right)=1-\dfrac{\sum_{n=2}^{N}\min\left\{ n-1,M\right\} }{\left(N-1\right)M}.
\]
\end{singlespace}

\begin{singlespace}
Now note that if $M>N-1$, then 
\[
\min\left\{ n-1,M\right\} =n-1,\:2\le n\le N,
\]
 so that
\[
\sum_{n=2}^{N}\min\left\{ n-1,M\right\} =\sum_{n=2}^{N}\left(n-1\right)=\dfrac{\left(N-1\right)N}{2}
\]
\[
\Longrightarrow\delta\left(N,M\right)\ge\beta\left(N,M\right)=1-\dfrac{N}{2M},
\]
where $\beta\left(N,M\right)\asymp1-N/\left(2M\right)$ as $N,M\rightarrow\infty$.

On the other hand, if $M\leq N-1$, then
\[
\min\left\{ n-1,M\right\} =\begin{cases}
n-1, & 2\le n\le M+1\\
M, & M+2\leq n\le N
\end{cases},
\]
so that
\[
\sum_{n=2}^{N}\min\left\{ n-1,M\right\} =\sum_{n=2}^{M+1}\left(n-1\right)+\sum_{n=M+2}^{N}M
\]
\[
=\dfrac{M\left(M+1\right)}{2}+\left(N-M-1\right)M=M\left(N-\dfrac{M+1}{2}\right)
\]
\[
\Longrightarrow\delta\left(N,M\right)\ge\beta\left(N,M\right)=1-\dfrac{N-\left(M+1\right)/2}{N-1}=\dfrac{M-1}{2\left(N-1\right)},
\]
where $\beta\left(N,M\right)\asymp M/\left(2N\right)$ as $N,M\rightarrow\infty$.
\end{singlespace}
\end{proof}
Clearly, the asymptotic lower bound $\beta\left(N,M\right)$ can converge
to different quantities in the limit, depending on $\underset{N,M\rightarrow\infty}{\lim}M/N$.
If $M/N\rightarrow0$, then $\beta\left(N,M\right)\rightarrow0$.
At the other extreme, if $M/N\rightarrow\infty$, then $\beta\left(N,M\right)\rightarrow1$.
For values of $\underset{N,M\rightarrow\infty}{\lim}M/N$ between
$0$ and $\infty$, $\beta\left(N,M\right)$ converges to values between
$0$ and $1$, respectively. Most saliently, if $M/N\rightarrow1$,
then $\beta\left(N,M\right)\rightarrow1/2$.

\subsection{Shifted Lattice Points}

\begin{singlespace}
\phantom{}

\medskip{}

\end{singlespace}

\begin{singlespace}
\noindent For variable $n\in\mathbb{Z}_{\geq1}$, $\kappa\in\mathcal{K}$,
and rationally shifted lattice points $\widetilde{m}\in\mathcal{L}_{\kappa}^{\left(+\right)}\left(\infty\right)$
or $\widetilde{m}\in\mathcal{L}_{\kappa}^{\left(-\right)}\left(\infty\right)$,
we define the bivariate densities
\[
\delta_{\kappa}^{\left(+\right)}\left(N,M\right)=\dfrac{\#\left(\left(n,\widetilde{m}\right)\in\left\{ 1,2,\ldots,N\right\} \times\mathcal{L}_{\kappa}^{\left(+\right)}\left(M\right):\Gamma^{\left(n\right)}\left(\widetilde{m}\right)\textrm{ is transcendental}\right)}{N\left(M+1\right)}
\]
and
\[
\delta_{\kappa}^{\left(-\right)}\left(N,M\right)=\dfrac{\#\left(\left(n,\widetilde{m}\right)\in\left\{ 1,2,\ldots,N\right\} \times\mathcal{L}_{\kappa}^{\left(-\right)}\left(M\right):\Gamma^{\left(n\right)}\left(\widetilde{m}\right)\textrm{ is transcendental}\right)}{N\left(M+1\right)},
\]
respectively. The following result provides a lower bound for $\delta_{\kappa}^{\left(+\right)}\left(N,M\right)$
and $\delta_{\kappa}^{\left(-\right)}\left(N,M\right)$ by leveraging
Theorem 2 to construct an upper bound for the number of algebraic
values among pairs $\left(n,\widetilde{m}\right)\in\left\{ 1,2,\ldots,N\right\} \times\mathcal{L}_{\kappa}^{\left(+\right)}\left(M\right)$
and $\left(n,\widetilde{m}\right)\in\left\{ 1,2,\ldots,N\right\} \times\mathcal{L}_{\kappa}^{\left(-\right)}\left(M\right)$.
\end{singlespace}
\begin{thm}
For $N\in\mathbb{Z}_{\geq1}$, $M\in\mathbb{Z}_{\geq0}$, and $\kappa\in\mathcal{K}$:
\[
\mathit{(i)}\;\delta_{\kappa}^{\left(+\right)}\left(N,M\right)\ge\widetilde{\beta}\left(N,M\right)=\begin{cases}
M/\left(2N\right), & M+1\leq N\\
1-\left(N+1\right)/\left[2\left(M+1\right)\right], & M+1>N
\end{cases}
\]
and
\[
\mathit{(ii)}\;\delta_{\kappa}^{\left(-\right)}\left(N,M\right)\ge\widetilde{\beta}\left(N,M\right)=\begin{cases}
M/\left(2N\right), & M+1\leq N\\
1-\left(N+1\right)/\left[2\left(M+1\right)\right], & M+1>N
\end{cases},
\]
where
\[
\widetilde{\beta}\left(N,M\right)\asymp\begin{cases}
M/\left(2N\right), & M/N\leq1\\
1-N/\left(2M\right), & M/N>1
\end{cases}
\]
as $N,M\rightarrow\infty$.
\end{thm}
\begin{proof}
\begin{singlespace}
\phantom{}
\end{singlespace}

\begin{singlespace}
\medskip{}
\end{singlespace}

\begin{singlespace}
\noindent (i), (ii) For a given $n\in\mathbb{Z}_{\geq1}$, we know
from Theorem 2 that the $M+1$ integers $m\in\left\{ 0,1,\ldots,M\right\} $
yield at most $\min\left\{ n,M+1\right\} $ algebraic values of $\Gamma^{\left(n\right)}\left(\widetilde{m}\right)$
for either $\widetilde{m}\in\mathcal{L}_{\kappa}^{\left(+\right)}\left(M\right)$
or $\widetilde{m}\in\mathcal{L}_{\kappa}^{\left(-\right)}\left(M\right)$,
and therefore at least $M+1-\min\left\{ n,M+1\right\} $ transcendental
values. Summing the latter expression over $n\in\left\{ 1,2,\dots,N\right\} $
then gives a (weak) lower bound for the total number of transcendental
values,
\[
N\left(M+1\right)\delta_{\kappa}^{\left(\bullet\right)}\left(N,M\right)\ge\sum_{n=1}^{N}\left(M+1-\min\left\{ n,M+1\right\} \right)=N\left(M+1\right)-\sum_{n=1}^{N}\min\left\{ n,M+1\right\} 
\]
\[
\Longrightarrow\delta_{\kappa}^{\left(\bullet\right)}\left(N,M\right)\ge\widetilde{\beta}\left(N,M\right)=1-\dfrac{\sum_{n=1}^{N}\min\left\{ n,M+1\right\} }{N\left(M+1\right)},
\]
where $\delta_{\kappa}^{\left(\bullet\right)}$ represents both $\delta_{\kappa}^{\left(+\right)}$
and $\delta_{\kappa}^{\left(-\right)}$ in turn. 
\end{singlespace}

\begin{singlespace}
Now note that if $M+1>N$, then 
\[
\min\left\{ n,M+1\right\} =n,\:1\le n\le N,
\]
 so that
\[
\sum_{n=1}^{N}\min\left\{ n,M+1\right\} =\sum_{n=1}^{N}n=\dfrac{N\left(N+1\right)}{2}
\]
\[
\Longrightarrow\delta_{\kappa}^{\left(\bullet\right)}\left(N,M\right)\ge\widetilde{\beta}\left(N,M\right)=1-\dfrac{N+1}{2\left(M+1\right)},
\]
where $\widetilde{\beta}\left(N,M\right)\asymp1-N/\left(2M\right)$
as $N,M\rightarrow\infty$.

On the other hand, if $M+1\leq N$, then
\[
\min\left\{ n,M+1\right\} =\begin{cases}
n, & 1\le n\le M+1\\
M+1, & M+2\leq n\le N
\end{cases},
\]
so that
\[
\sum_{n=1}^{N}\min\left\{ n,M+1\right\} =\sum_{n=1}^{M+1}n+\sum_{n=M+2}^{N}\left(M+1\right)
\]
\[
=\dfrac{\left(M+1\right)\left(M+2\right)}{2}+\left(N-M-1\right)\left(M+1\right)=\left(M+1\right)\left(N-\dfrac{M}{2}\right)
\]
\[
\Longrightarrow\delta_{\kappa}^{\left(\bullet\right)}\left(N,M\right)\ge\widetilde{\beta}\left(N,M\right)=1-\dfrac{N-M/2}{N}=\dfrac{M}{2N},
\]
where $\widetilde{\beta}\left(N,M\right)\asymp M/\left(2N\right)$
as $N,M\rightarrow\infty$.
\end{singlespace}
\end{proof}
As with the asymptotic lower bound given by Theorem 3, $\widetilde{\beta}\left(N,M\right)$
can converge to different quantities in the limit, depending on $\underset{N,M\rightarrow\infty}{\lim}M/N$.
In particular: if $M/N\rightarrow0$, then $\widetilde{\beta}\left(N,M\right)\rightarrow0$;
if $M/N\rightarrow\infty$, then $\widetilde{\beta}\left(N,M\right)\rightarrow1$;
and if $\underset{N,M\rightarrow\infty}{\lim}M/N=L\in\left(0,\infty\right)$,
then $\widetilde{\beta}\left(N,M\right)\rightarrow b\in\left(0,1\right)$
(with $b=1/2$ when $L=1$).

\section{Conclusion}

\begin{singlespace}
\noindent The present study investigated the transcendence of $\Gamma^{\left(n\right)}\left(m\right)$
at positive lattice points $m\in\mathcal{L}\left(\infty\right)$ and
$\Gamma^{\left(n\right)}\left(\widetilde{m}\right)$ at rationally
shifted lattice points $\widetilde{m}\in\mathcal{L}_{\kappa}^{\left(+\right)}\left(\infty\right)$
or $\widetilde{m}\in\mathcal{L}_{\kappa}^{\left(-\right)}\left(\infty\right)$
for $\kappa\in\left(0,1\right)\cap\mathbb{Q}$ such that $\Gamma\left(\kappa\right)$
is transcendental.
\end{singlespace}

\begin{singlespace}
For fixed $n$, we derived strong upper bounds for the total number
of algebraic $\Gamma^{\left(n\right)}\left(m\right)$ and $\Gamma^{\left(n\right)}\left(\widetilde{m}\right)$,
showing that the maximum number of algebraic $\Gamma^{\left(n\right)}\left(m\right)$
is $n-1$ and the maximum number of algebraic $\Gamma^{\left(n\right)}\left(\widetilde{m}\right)$
is $n$ for each one-sided shifted lattice (i.e., $\widetilde{m}\in\mathcal{L}_{\kappa}^{\left(+\right)}\left(\infty\right)$
or $\widetilde{m}\in\mathcal{L}_{\kappa}^{\left(-\right)}\left(\infty\right)$).
These results imply that the density of transcendental $\Gamma^{\left(n\right)}\left(m\right)$
among $m\in\mathcal{L}\left(M\right)$ has lower bound $\beta_{n}\left(M\right)=1-\min\left\{ n-1,M\right\} /M\asymp1-\left(n-1\right)/M\rightarrow1$
as $M\rightarrow\infty$, whereas the density of transcendental $\Gamma^{\left(n\right)}\left(\widetilde{m}\right)$
among either $\widetilde{m}\in\mathcal{L}_{\kappa}^{\left(+\right)}\left(M\right)$
or $\widetilde{m}\in\mathcal{L}_{\kappa}^{\left(-\right)}\left(M\right)$
has lower bound $\widetilde{\beta}_{n}\left(M\right)=1-\min\left\{ n,M+1\right\} /\left(M+1\right)\asymp1-n/M\rightarrow1$
as $M\rightarrow\infty$.

We then allowed $n$ to vary, evaluating the bivariate densities of
both transcendental $\Gamma^{\left(n\right)}\left(m\right)$ among
$\left(n,m\right)\in\left\{ 2,3,\ldots,N\right\} \times\mathcal{L}\left(M\right)$
and transcendental $\Gamma^{\left(n\right)}\left(\widetilde{m}\right)$
among $\left(n,\widetilde{m}\right)\in\left\{ 1,2,\ldots,N\right\} \times\mathcal{L}_{\kappa}^{\left(+\right)}\left(M\right)$
and $\left(n,\widetilde{m}\right)\in\left\{ 1,2,\ldots,N\right\} \times\mathcal{L}_{\kappa}^{\left(-\right)}\left(M\right)$.
In both cases, the analysis yielded lower bounds -- denoted by $\beta\left(N,M\right)$
and $\widetilde{\beta}\left(N,M\right)$, respectively -- that depended
approximately on the ratio $M/N$. In particular, $\beta\left(N,M\right)\asymp\widetilde{\beta}\left(N,M\right)\asymp M/\left(2N\right)$
for $M/N\leq1$ and $\beta\left(N,M\right)\asymp\widetilde{\beta}\left(N,M\right)\asymp1-N/\left(2M\right)$
for $M/N>1$.
\end{singlespace}

Given the conspicuous gap between the relative strengths of (i) the
lower bound for fixed $m$ in Theorem 2 of Powers {[}3{]} (i.e., $\beta\left(N\right)\asymp1/\sqrt{N}\;\textrm{as}\;N\rightarrow\infty$),
and (ii) the lower bound for fixed $n$ in Corollary 1 of the present
article (i.e., $\beta_{n}\left(M\right)\asymp1-\left(n-1\right)/M\;\textrm{as}\;M\rightarrow\infty$),
the most obvious direction for further research is to improve the
former result. Undoubtedly, this will require the application of new
techniques, possibly assisted by a deeper understanding of how recursive
relations among the $\Gamma^{\left(n\right)}\left(m\right)$ for the
separate $n$ and $m$ dimensions are interrelated.

\medskip{}
\medskip{}

\end{document}